\documentclass[11pt]{amsart}
\usepackage[margin=1.0in]{geometry}
\usepackage{amsmath,amssymb,amsfonts,graphicx,color,fancyhdr,psfrag,comment,enumerate}
\usepackage[latin1]{inputenc}
\usepackage[hyperpageref]{backref}
\usepackage[colorlinks=true, pdfstartview=FitV, linkcolor=blue, citecolor=blue, urlcolor=blue]{hyperref}
\usepackage[capitalise,noabbrev]{cleveref}
\usepackage{tcolorbox,longfbox}
\allowdisplaybreaks
\usepackage{epstopdf}
\usepackage{mathrsfs}
\usepackage{epsfig}
\usepackage{hyperref}

\usepackage{ifthen}
\usepackage{float}
\usepackage{enumerate}
\usepackage{mathtools}
\usepackage[bottom]{footmisc}
\usepackage{tikz}
\usepackage{comment}
\usetikzlibrary{matrix,shapes,arrows,positioning,chains}
\usepackage{amssymb,amsmath,amsfonts}
\usepackage{bbm}

\usepackage{cleveref}
\allowdisplaybreaks

\DeclareMathOperator\supp{supp}
\DeclareMathOperator\sloc{sloc}

\newtheorem{lemma}{Lemma}[section]

\numberwithin{equation}{section}
\newtheorem{theorem}{Theorem}[section]
\newtheorem{claim}{Claim}[]
\newtheorem{proposition}[theorem]{Proposition}

\newtheorem{corollary}[theorem]{Corollary}

\newcommand{\R}{\mathbb{R}}
\newcommand{\C}{\mathbb{C}}
\newcommand{\Z}{\mathbb{Z}}

\begin{document}
\title[Sub-Laplacians with drift on M\'etivier groups]{On Spectral multiplier theorem for sub-Laplacians with drift on M\'etivier groups}

\author[N. Garg, J. Singh]
{Nishta Garg \and Joydwip Singh} 

\address[N. Garg]{Department of Mathematics, Indian Institute of Science Education and Research Bhopal, Bhopal--462066, Madhya Pradesh, India.}
\email{nishta21@iiserb.ac.in}

\address[J. Singh]{Department of Mathematical Sciences, Indian Institute of Science Education and Research Mohali, Mohali--140306, Punjab, India.}
\email{joydwipsingh@iisermohali.ac.in}



\begin{abstract}
In this paper, we prove a spectral multiplier theorem for sub-Laplacians with drift on M\'etivier groups. We improve the result of [Martini, Ottazzi and Vallarino, Rev. Mat. Iberoam, 2019] in case of M\'etivier groups, by reducing the required smoothness condition on the multiplier function from homogeneous dimension to the topological dimension of the underlying group.
\end{abstract}

\maketitle

\section{Introduction}
Let $G$ be a connected, simply connected, two-step stratified Lie group with Lie algebra $\mathfrak{g}$, that is, the algebra $\mathfrak{g}$ can be decomposed into two non-trivial subspaces as $\mathfrak{g}_1 \oplus \mathfrak{g}_2$ such that $[\mathfrak{g}_1, \mathfrak{g}_1] = \mathfrak{g}_2$ and $[\mathfrak{g}_1, \mathfrak{g}_2] = \{0\}$. Let $\dim \mathfrak{g}_1 = d_1$ and $\dim \mathfrak{g}_2 = d_2$. One often calls $\mathfrak{g}_1$ and $\mathfrak{g}_2$ as the first and second layer of $\mathfrak{g}$, respectively. Let $\{X_1, \ldots, X_{d_1}, T_1, \ldots, T_{d_2}\}$ be a basis of $\mathfrak{g}$ and $\langle \cdot, \cdot \rangle$ denote an inner product on $\mathfrak{g}$ with respect to which the basis becomes an orthonormal basis. The inner product $\langle \cdot, \cdot \rangle$ on $\mathfrak{g}$ induces a norm $|\cdot|$ on $\mathfrak{g}_2^*$, the dual of $\mathfrak{g}_2$. For each $\lambda \in \mathfrak{g}_2^*$, there exists a skew symmetric bilinear form $\omega_{\lambda} : \mathfrak{g}_1 \times \mathfrak{g}_1 \to \mathbb{R}$ such that
\begin{align*}
    \omega_{\lambda}(x, x') := \lambda([x, x']) .
\end{align*}
Corresponding to $\omega_{\lambda}$, there is a skew symmetric endomorphism $J_{\lambda}$ on $\mathfrak{g}_1$ which satisfies
\begin{align*}
    \omega_{\lambda}(x, x') = \langle J_{\lambda} x, x' \rangle .
\end{align*}
The group $G$ is said to be a M\'etivier group if $J_{\lambda}$ is invertible for all $\lambda \neq 0$ and Heisenberg-type group if $J_{\lambda}^2 = -|\lambda|^2 \, \text{id}_{\mathfrak{g}_1}$. Note that the class of M\'etivier groups is strictly larger than the class of Heisenberg-type groups, see \cite{Muller_Seeger_Singular_Spherical_maximal_operator_Nilpotent_2004}.

Denote $Q:=d_1+2d_2$ and $d:=d_1+d_2$ to be the homogeneous and topological dimension of $G$ respectively. Note that the basis $\{X_1, \ldots, X_{d_1}, T_1, \ldots, T_{d_2}\}$ of $\mathfrak{g}$, can be identified with the left-invariant vector fields via the Lie derivative. Corresponding to the basis of the first layer vector fields $\{X_1, \ldots, X_{d_1}\}$, the sub-Laplacian $\mathcal{L}$ on M\'etivier groups is defined as a second-order left-invariant differential operator, given by
\begin{align*}
    \mathcal{L} = -(X_1^2 + \cdots + X_{d_1}^2) .
\end{align*}
The sub-Laplacian $\mathcal{L}$ turns out to be positive and essentially self-adjoint on $L^2(G)$, where $G$ is endowed with the Haar measure. Since the group $G$ is nilpotent Lie group, the exponential map $\exp: \mathfrak{g} \to G$ is a global diffeomorphism, therefore using this exponential map one may identify $G$ with its Lie algebra $\mathfrak{g} = \mathfrak{g}_1 \oplus \mathfrak{g}_2$, and via the chosen basis $\{X_1, \ldots, X_{d_1}, T_1, \ldots, T_{d_2}\}$ of $\mathfrak{g}$, one may also identify $\mathfrak{g}$ with $\mathbb{R}_x^{d_1} \times \mathbb{R}_t^{d_2}$, that is, $G \cong \mathfrak{g} \cong \mathbb{R}_x^{d_1} \times \mathbb{R}_t^{d_2}$. Also the Haar measure of $G$ is given by the Lebesgue measure $dx \, dt$ of $\mathbb{R}^{d_1} \times \mathbb{R}^{d_2}$.

Now, for a left-invariant vector field $X \in \mathfrak{g}$, consider 
\begin{align}
    \label{eq_sub-Laplacian_with_drift}
\mathcal{L}_{X} := \mathcal{L} - X = - \sum_{j=1}^{d_1} X_{j}^{2} - X.
\end{align}
In \cite{Hebisch_Mauceri_Meda_spectral_multipliers_drift_2015}, it was proved that $\mathcal{L}_{X}$ is symmetric in $L^2(d\mu)$ for some positive measure $\mu$ on $G$ if and only if there is a positive character $\chi$ on the group $G$ such that $d\mu = \chi \, dx \, dt$ and $X = \sum_{j=1}^{d_1} d\chi_{0}(X_j) X_{j}$ where $0:=(0,0)$ is the identity element of $G$. From now on, we shall write $d\mu_{X} := \chi \, dx \, dt$. It turns out that the sub-Laplacian with drift $\mathcal{L}_{X}$ associated with the drift $X$ is essentially self-adjoint on $L^{2}(d\mu_{X})$ and the spectrum is contained in $[b_X^2, \infty)$, where $b_{X} = |X|/2$ with $|X| = \left(\sum_{j=1}^{d_1}|d\chi_{0}(X_{j})|^2\right)^{1/2}$. In this paper, we shall always assume $X \neq 0$.

The sub-Laplacian with drift is a very well studied topic in the literature, for example regarding the heat kernel estimates of $\mathcal{L}_X$, see \cite{Alexopoulos_Heat_kernel_Drift_2002, Dungey_Heat_Kernel_Drift_2005}, $L^p$-boundedness of Riesz transform associated to $\mathcal{L}_X$, see \cite{Lohoue_Mustapha_Riesz_Transform_2004, Li_Sjogren_Sub_Laplacian_Drift_2021}, $L^p$-boundedness of spectral multipliers \cite{Hebisch_Mauceri_Meda_spectral_multipliers_drift_2015, Martini_Ottazzi_Vallarino_spectral_multipliers_drift_2019} and references therein. In this paper, we study the boundedness of spectral multipliers associated with the sub-Laplacian with drift $\mathcal{L}_{X}$ on M\'etivier groups.

For any bounded Borel function $F$ defined on $[b_{X}^{2}, \infty)$, from the spectral theorem of $\mathcal{L}_{X}$ we have
$$
F(\mathcal{L}_{X}) = \int_{b_{X}^2}^{\infty} F(\lambda) \, dE_{\lambda},
$$
where $E_{\lambda}$ denotes the spectral measure of $\mathcal{L}_{X}$. It is well known from the spectral theory that the operator $F(\mathcal{L}_{X})$ is bounded on $L^{2}(d\mu_{X})$. Note that, the bounded function $F$ is called multiplier and the corresponding operator $F(\mathcal{L}_{X})$ as spectral multiplier. But for $p\neq 2$, it is known that, only boundedness of the multiplier $F$ is not sufficient to get the $L^p$-boundedness of $F(\mathcal{L}_{X})$. Therefore, it is natural to ask under what additional condition on the multiplier $F$, the spectral multiplier $F(\mathcal{L}_{X})$ initially defined on $L^{2}(d\mu_{X}) \cap L^{p}(d\mu_{X})$ extends to a bounded operator on $L^{p}(d\mu_{X})$ for some $1 \leq p \leq \infty$ and $p \neq 2$. 

In this article, we are interested in studying the $L^p$-boundedness of the spectral multiplier $F(\mathcal{L}_{X})$. For amenable groups $G$, Hebisch, Mauceri and Meda \cite{Hebisch_Mauceri_Meda_spectral_multipliers_drift_2015} proved that for $p \in [1,\infty]\setminus \{2\}$ every $L^{p}(d\mu_{X})$ spectral multiplier $F$ of $\mathcal{L}_{X}$ extends to a bounded holomorphic function on a parabolic region $P_{X,p}$ in the complex plane, given by
\begin{align*}
    P_{X,p} &:= \Big\{ x+iy \in \mathbb{C} : x>\frac{y^2}{4 b_X^2 \sin^2 \phi_p^*} + b_X^2 \cos^2\phi_p^* \Big\} ,
\end{align*}
where $\phi_p^* = \arcsin|2/p-1|$. Also, in the case when $G$ has polynomial growth, they proved a sufficient condition for the holomorphic function $F$ defined on the same parabolic region $P_{X,p}$, to be a spectral multiplier on $L^p(d\mu_{X})$ for $1 < p < \infty$. 

Recently, Martini, Ottazzi and Vallarino \cite{Martini_Ottazzi_Vallarino_spectral_multipliers_drift_2019} further improved the work of Hebisch, Mauceri and Meda \cite{Hebisch_Mauceri_Meda_spectral_multipliers_drift_2015} by means of giving the sufficient condition in terms of the $L^q$-norm instead of the pointwise condition on the spectral multiplier $F$ and also reducing the number of derivatives as well as making it dependent on $p$. In order to state their result, first we need some notations.

Let $\psi$ be a bump function supported in $[1/4,4]$ such that it satisfies
\begin{align}
\label{Definition of psi}
    \sum_{j \in \mathbb{Z}} \psi(2^j t) = 1 \quad \text{for all} \quad t>0 .
\end{align}

For $s \in \mathbb{R}$ and $1\leq p \leq \infty$, let $L_s^p$ denote the $L^p$-Sobolev space of fractional order $s$ on $\mathbb{R}$. Then for $s \geq 0$ and $F: \mathbb{R} \to \mathbb{C}$, we define the scale-invariant local Sobolev norm by
\begin{align*}
    \|F\|_{L^2_{s, \sloc}} &:= \sup_{t>0} \|F(t\cdot) \psi\|_{L^2_s} .
\end{align*}

We also define the weighted Sobolev spaces on $\mathbb{R}$, denoted by $L_{s,r}^{p, \kappa}$, to consists of all functions $F$ such that the following norm is finite: for all $s, \kappa \in \mathbb{R}$, $1<p<\infty$ and $1\leq r \leq \infty$ and $F: \mathbb{R} \to \mathbb{C}$,
\begin{align*}
    \|F\|_{L_{s,r}^{p, \kappa}} &= \left\{\begin{array}{ll}
       \left[\sum_{j \in \mathbb{N}} \left[2^{j(\kappa+1/p)}\|F(2^j \cdot)\psi_{(j)}\|_{L_s^p} \right]^r \right]^{1/r}  & \quad \text{if} \quad r <\infty \\
       \sup_{j \in \mathbb{N}} 2^{j(\kappa+1/p)}\|F(2^j \cdot)\psi_{(j)}\|_{L_s^p} & \quad \text{if} \quad r=\infty ,
    \end{array}\right.
\end{align*}
where $\psi_{(j)} = \sum_{\epsilon=\pm 1} \psi(\epsilon \cdot)$ if $j>0$ and $\psi_{(0)} = \sum_{\epsilon=\pm 1, k \in \mathbb{N}} \psi(\epsilon 2^k \cdot)$.

\medskip
For $\Theta \in (0, \infty)$, we define a complex strip $S_{\Theta}$ denoted by $S_{\Theta} := \{z \in \mathbb{C} : |\text{Im}\, z| < \Theta\}$. Now for all $p \in (1, \infty)$, $r \in [1, \infty]$, $s\in (1/p, \infty)$ and $\kappa \in \mathbb{R}$, we denote
\begin{align*}
    \mathscr{L}_{s,r}^{p, \kappa, \Theta} := \{F(\cdot \pm i \Theta) \in L_{s,r}^{p, \kappa} : F: \overline{S_{\Theta}} \to \mathbb{C}\  & \text{even, continuous, holomorphic on}\ S_{\Theta} \\
    &\text{and have at most polynomial growth}\}.
\end{align*}

For $\Theta =0$, we denote $\mathscr{L}_{s,r}^{p, \kappa, 0}$ to be the space of even functions which belongs to $L_{s,r}^{p, \kappa}$. Then for any $\Theta \in [0, \infty)$, the spaces $\mathscr{L}_{s,r}^{p, \kappa, \Theta}$ become a Banach space with the following norm given by
\begin{align*}
    \|F\|_{\mathscr{L}_{s,r}^{p, \kappa, \Theta}} &:= \|F(\cdot \pm i \Theta)\|_{L_{s,r}^{p, \kappa}} .
\end{align*}

For any $N \in \mathbb{N}$, let $H^{\infty}(S_{\Theta}; N)$ denote the space of bounded even holomorphic functions $F$ on $S_{\Theta}$ and continuous on $\overline{S_{\Theta}}$ with all their derivatives up to the order $N$, such that
\begin{align*}
    \|F\|_{H^{\infty}(S_{\Theta}; N)} &:= \max_{j\in \{0, \ldots, N\}} \sup_{z \in \overline{S_{\Theta}}} (1+|z|)^j |F^{(j)}(z)| < \infty .
\end{align*}

Note that by a change of variable $z \mapsto b_X^2 + z^2$, the strip $S_{\Theta}$ becomes the parabolic region $P_{X,p}$ defined earlier, where $\Theta = |2/p - 1|b_{X}$. Instead of giving a condition on the holomorphic function defined on the parabolic region, it is sometimes convenient to express the condition on the holomorphic function defined on this strip $S_{\Theta}$. 

Now we are in a position to state the result of \cite{Martini_Ottazzi_Vallarino_spectral_multipliers_drift_2019}. Also, we use the local Hardy $\mathfrak{h}^1(d\mu_X)$ and $\mathfrak{bmo}(d\mu_X)$ spaces to state this result, and for their definition one can refer to \cite{Goldberg_hardy_space_79, Meda_Volpi_Goldberg_type_17, Martini_Ottazzi_Vallarino_spectral_multipliers_drift_2019}. In 2019, Martini, Ottazzi, and Vallarino \cite{Martini_Ottazzi_Vallarino_spectral_multipliers_drift_2019} proved the following multiplier theorem for the sub-Laplacian with drift $\mathcal{L}_{X}$ on non-compact Lie groups. Here we only state their result for two-step stratified Lie groups, but they have proved the result in much more generality. Let us set $\mathfrak{D}_X := \sqrt{\mathcal{L}_X-b_X^2}$.

\begin{theorem}\cite[Theorem 3.1, Lemma 4.1]{Martini_Ottazzi_Vallarino_spectral_multipliers_drift_2019}
\label{Theorem: Martini Ottazzi Vallarino theorem}
    Let $G$ be a two-step stratified Lie group. Let $p \in [1,\infty] \setminus \{2\}$, $q \in [2, \infty)$ such that $1/q=|2/p-1|/2$. Set $\Theta_{X, p} = |2/p - 1| b_{X}$. Suppose
\begin{align*}
    F \in \mathscr{L}_{s,\infty}^{q, -1/q, \Theta_{X, p}} \quad \text{for some} \quad s > |1/p - 1/2|(Q+1) ,
\end{align*}
then
\begin{enumerate}
    \item If $p \in (1,\infty)$, then $F(\mathfrak{D}_X)$ extends to a bounded operator on $L^p(d\mu_X)$.
    \medskip
    \item If $p =1$, then $F(\mathfrak{D}_X)$ extends to a bounded operator on $\mathfrak{h}^1(d\mu_X)$ to $L^1(d\mu_X)$.
    \medskip
    \item If $p =\infty$, then $F(\mathfrak{D}_X)$ extends to a bounded operator on $L^{\infty}(d\mu_X)$ to $\mathfrak{bmo}(d\mu_X)$.
\end{enumerate}

\end{theorem}

Note that, without drift, the $L^p$-boundedness of the spectral multiplier associated to the sub-Laplacian $\mathcal{L}$ on two-step stratified Lie group $G$ has been studied by Christ \cite{Christ_Spectral_multiplier_Nilpotent_groups_1991} and independently by Mauceri and Meda \cite{Mauceri_Meda_Multipliers_Stratified_groups_1990}. There the smoothness condition on the multiplier was expressed in terms of the homogeneous dimension $Q$ of the group $G$. However, for the case of Heisenberg-type groups, M\"uller and Stein \cite{Spectral_multipliers_Stein_Muller}, and independently Hebisch \cite{Hebisch_Spectral_Multiplier_Heisenberg_1993}, proved that the smoothness condition on the multiplier function can be further pushed down from the homogeneous dimension $Q$ to the topological dimension $d$ of Heisenberg-type groups, and these results also turn out to be optimal. After that, a lot of study has been conducted in various setups to obtain the optimal threshold for the $L^p$-boundedness of the spectral multiplier theorem, see \cite{Cowling_Klima_Sikora_Kohn_Laplacian_On_Sphere_2011}, \cite{Martini_Muller_Multiplier_Grushin_2014}, \cite{Martini_Lie_groups_Polynomial_Growth_2012}, \cite{Martini_Spectral_Multiplier_Heisenberg_Reiter_2015}, \cite{Casarino_Ciatti_Martini_Grushin_Sphere_2019} and references therein.

\emph{A natural question, therefore, is whether in Theorem \ref{Theorem: Martini Ottazzi Vallarino theorem}, the smoothness condition on $F$ can be pushed down to $s > d|1/p- 1/2|$, where $d$ is the topological dimension of the group $G$.} At this moment, we do not know whether such an improvement holds or not. In this paper, however, working on the setting of M\'etivier groups, we are able to reduce the smoothness requirements to $s > (d+1)|1/p- 1/2|$. More precisely, we prove the following result.

\begin{theorem}
\label{Theorem: Metivier drift main theorem}
    Let $G$ be a M\'etivier group. Let $p \in [1,\infty] \setminus \{2\}$, $q \in [2, \infty)$ such that $1/q=|2/p-1|/2$. Set $\Theta_{X, p} = |2/p - 1| b_{X}$. Suppose
\begin{align*}
    F \in \mathscr{L}_{s,\infty}^{q, -1/q, \Theta_{X, p}} \quad \text{for some} \quad s > |1/p - 1/2|(d+1) ,
\end{align*}
then
\begin{enumerate}
    \item If $p \in (1,\infty)$, then $F(\mathfrak{D}_X)$ extends to a bounded operator on $L^p(d\mu_X)$.
    \medskip
    \item If $p =1$, then $F(\mathfrak{D}_X)$ extends to a bounded operator on $\mathfrak{h}^1(d\mu_X)$ to $L^1(d\mu_X)$.
    \medskip
    \item If $p =\infty$, then $F(\mathfrak{D}_X)$ extends to a bounded operator on $L^{\infty}(d\mu_X)$ to $\mathfrak{bmo}(d\mu_X)$.
\end{enumerate}

\end{theorem}

Note that in the above theorem the smoothness condition on the multiplier function $F \in \mathscr{L}_{s,\infty}^{q, -1/q, \Theta_{X, p}}$ is described in terms of the $L^q$-Sobolev spaces for $q \in [2, \infty)$. Instead, if we want to give pointwise condition on the the multiplier $F$, then in this regard we have following result.

\begin{corollary}
\label{Corollary: Metivier drift corollary of main}
Let $G$ be a M\'etivier group. For $p \in [1,\infty] \setminus \{2\}$, set $\Theta_{X, p} = |2/p - 1| b_{X}$. Suppose
\begin{align*}
    F \in H^{\infty}(S_{\Theta_{X,p}}, N) \quad \text{for some} \quad N \in \mathbb{N} \quad \text{and} \quad N>|1/p - 1/2|(d+1) ,
\end{align*}
then the same conclusions (1), (2) and (3) of Theorem \ref{Theorem: Metivier drift main theorem} hold.
    
\end{corollary}

The above corollary immediately follows from the following fact (see Lemma 2.10 (vii) in \cite{Martini_Ottazzi_Vallarino_spectral_multipliers_drift_2019})
\begin{align*}
    \|F\|_{\mathscr{L}_{s,\infty}^{q, -1/q, \Theta_{X, p}}} \lesssim \|F\|_{H^{\infty}(S_{\Theta_{X,p}}, N)} ,
\end{align*}
and by choosing $s$ such that $N > s > |1/p - 1/2|(d+1)$ and then applying Theorem \ref{Theorem: Metivier drift main theorem}.

\medskip
Note that Theorem \ref{Theorem: Metivier drift main theorem} and Corollary \ref{Corollary: Metivier drift corollary of main} are stated in terms of $\mathfrak{D}_X$. Using the change of variable $z \mapsto b_X^2 + z^2$, actually we can obtain similar results for the operator $\mathcal{L}_X$, for more details see \cite[Remark 3.3]{Martini_Ottazzi_Vallarino_spectral_multipliers_drift_2019}. Let us again emphasize the fact that in both of our result (Theorem \ref{Theorem: Metivier drift main theorem} and Corollary \ref{Corollary: Metivier drift corollary of main}) the smoothness condition on the multiplier $F$ is expressed in terms of the topological dimension $d$ of the underlying M\'etivier group $G$. Consequently, in the case of M\'etivier groups, our main result Theorem \ref{Theorem: Metivier drift main theorem} improves the required smoothness condition described in Theorem \ref{Theorem: Martini Ottazzi Vallarino theorem}, due to Martini, Ottazzi, and Vallarino \cite{Martini_Ottazzi_Vallarino_spectral_multipliers_drift_2019}.

\medskip
\textbf{Notations:}
Throughout the paper we use the standard notations. We shall use the notation $f \lesssim g$ to indicate $f \leq Cg$ for some $C > 0$, and whenever $f \lesssim g\lesssim f$, we shall write $f \sim g$. Also, we write $f \lesssim_{\epsilon} g$ when the implicit constant $C$ may depend on a parameter like $\epsilon$. Also, $ B := B(x, r)$ denote the ball with centered at $x$ and radius $r$ and the notation $\kappa B=B(x, \kappa r)$ stands for the concentric dilation of $B$ by $\kappa >0$. For a ball $B$, write $B^c$ to denote the complement of the ball $B$.

\medskip
\textbf{Organization:}
The paper is organized as follows. In Section \ref{Section: Preliminaries and kernel estimate} we collect some preliminary details about distance function, balls and character on M\'etivier group. After that we prove a first-layer weighted Plancherel estimate where the multiplier is not assumed to be compactly supported. Consequently we also discuss some $L^1$ and $L^2$ kernel estimates. Section \ref{Section: Proof of main theorem} is devoted to the proof of our main Theorem \ref{Theorem: Metivier drift main theorem}, upon assuming three claims. Finally in Section \ref{Section: Proof of claims} we give the proof of the mentioned three claims \ref{Assumption 1}, \ref{Assumption 2} and \ref{Assumption 3} in Section \ref{Section: Proof of main theorem}.


\section{Preliminaries and Kernel estimates}
\label{Section: Preliminaries and kernel estimate}
The spectral decomposition for the sub-Laplacian $\mathcal{L}$ on M\'etivier groups has been well studied and explicitly known \cite{Niedorf_Metivier_Group_2025, Martini_Muller_New_Class_Two_Step_Stratified_Groups_2014}. For a bounded Borel function $F: \R \to \C$, the operator $F(\sqrt{\mathcal{L}})$ possess convolution kernel $\mathcal{K}_{F(\sqrt{\mathcal{L}})}$, that is, 
\begin{align*}
    F(\sqrt{\mathcal{L}})f(x,u) &= f\ast \mathcal{K}_{F(\sqrt{\mathcal{L}})}(x,u) = \int_{G} f(y,s) \, \mathcal{K}_{F(\sqrt{\mathcal{L}})}((y,s)^{-1}(x,u)) \, d(y, s) .
\end{align*}
Let $F$ be a Schwartz class function, then the convolution kernel $\mathcal{K}_{F(\sqrt{\mathcal{L}})}$ is also a Schwartz class function and is given by 
\begin{align*}
\mathcal{K}_{F(\sqrt{\mathcal{L}})}(x,u) &= \tfrac{1}{(2 \pi)^{d_2}} \int_{\mathfrak{g}_{2,r}^{*}} \sum_{\mathbf{k} \in \mathbb{N}^{\Lambda}} F(\sqrt{\eta_{\mathbf{k}}^{\lambda}}) \prod_{n=1}^{\Lambda} \varphi_{k_n}^{(b_n^{\lambda}, r_n)}(R_{\lambda}^{-1} P_n^{\lambda}x) e^{i \langle \lambda, u \rangle} \, d \lambda,
\end{align*}
where $\Lambda \in \mathbb{N}$, $(b_1^{\lambda}, \ldots, b_{\Lambda}^{\lambda}) \in (0, \infty)^\Lambda$, $(r_1, \ldots, r_\Lambda) \in \mathbb{N}^\Lambda$, $\mathbf{k}=(k_1, \ldots, k_\Lambda) \in \mathbb{N}_0^\Lambda$, $\mathfrak{g}_{2,r}^{*}$ is the Zariski open subset of $\mathfrak{g}_{2}^{*}$, $\eta_{\mathbf{k}}^{\lambda} = \sum_{n=1}^{\Lambda} (2 k_n + r_n) b_n^{\lambda}$, the functions $\lambda \to b_n^{\lambda}$ are homogeneous of degree $1$, the functions $\lambda \to R_{\lambda}$ are Borel measurable on $\mathfrak{g}_{2,r}^{*}$ and homogeneous of degree $0$, the functions $\lambda \to P_{n}^{\lambda}$ are (componentwise) real analytic on $\mathfrak{g}_{2,r}^{*}$ and homogeneous of degree $0$ and $\varphi_{k}^{(\mu, m)}(z) = \mu^m L^{m-1}_k(\tfrac{1}{2}\mu |z|^2) e^{-\frac{1}{2}\mu |z|^2}$ for $z \in \mathbb{R}^{2m}$, $\mu>0$ is the $\mu$-rescaled Laguerre function, where $L^{m-1}_k$ denotes the $k$-th Laguerre polynomial of type $m-1$ (see \cite[Proposition 3.1]{Niedorf_Metivier_Group_2025}).

\medskip
Let $\varrho$ denote the Carnot-Carath\'eodory distance on $G$ and $B((x,u),r)$ denote the ball centered at $(x,u)$ and of radius $r$ with respect $\varrho$. Then the volume of a ball is given by
\begin{align*}
    |B((x,u),R)| \sim R^Q |B(0,1)|,
\end{align*}
where $|\cdot|$ denote the Lebesgue measure of $\mathbb{R}^{d_1} \times \mathbb{R}^{d_2}$ and $Q=d_1+2d_2$ is the homogeneous dimension of $G$. We also define the homogeneous norm on $G$ by
\begin{align*}
    \|(x,u)\| &:= |x| + |u|^{1/2}, \quad \quad \text{for} \quad (x,u) \in G .
\end{align*} 
It is well known that $\varrho((x,u), 0) \sim \|(x,u)\|$, see \cite[p. 8]{Folland_Stein_Hardy_space_homogeneous_group_1982}.

Recall that, $\chi$ is positive character on $G$, such that for $X \in \mathfrak{g}$ we have $X = \sum_{j=1}^{d_1} d\chi_{0}(X_j) X_{j}$. The following lemma tells about the pointwise growth of the character $\chi$ and $\mu_X$ measure of a ball.

\begin{lemma}\cite[Lemma 2.2]{Martini_Ottazzi_Vallarino_spectral_multipliers_drift_2019}
\label{Lemma: Pointwise estimate of character}
    For all $(x,u) \in G$, we have
    \begin{align*}
        \chi(x,u) &\lesssim e^{|X| \|(x,u)\|} .
    \end{align*}
    Moreover, for all $R \geq 1$, we have
    \begin{align*}
        \mu_X (B(0, R)) &\lesssim R^{Q-1} e^{|X| R} .
    \end{align*}
\end{lemma}

Let us record a useful Lemma, which provide estimates of the first-layer weight and homogeneous norms. This will be used later in our proofs.

\begin{lemma}\cite[Lemma 2.5]{Singh_2026_Stein_square_function}
\label{Lemma: Integration of distance and weight}
Let $R>0$. Then for all $\alpha+\beta>Q$ and all $0\leq \alpha <d_1$, we have
\begin{align*}
    \int_G \frac{d(x,u)}{(1+R\|(x,u)\|)^{\beta} (1+R|x|)^{\alpha}} &\leq C R^{-Q} .
\end{align*}
    
\end{lemma}


We now prove a first-layer weighted Plancherel estimate, which is one of the main ingredients in the proof of Theorem \ref{Theorem: Metivier drift main theorem}. The proof of the following proposition is similar to that of  \cite[Proposition 6.1]{Niedorf_Metivier_Group_2025}, but it is worth noting that here we do not impose any support condition on the multiplier function $F$, which is essential later in the proof of Claim \ref{Assumption 3}.

\begin{proposition}
\label{Proposition: First layer weighted Plancherel}
If $F: \mathbb{R} \to \mathbb{C}$ is a bounded Borel function, then for all $0\leq \alpha<d_2/2$ we have
\begin{align}
\label{Estimate: First layer weighted Plancherel for limited alpha}
\int_G | |x|^{\alpha} \mathcal{K}_{F(\sqrt{\mathcal{L}})}(x,u) |^2 \, d(x,u) &\leq C \int_{0}^{\infty} |F(\eta)|^2 \eta^{Q-2\alpha} \frac{d\eta}{\eta} .
\end{align}
\end{proposition}

\begin{proof}
Via interpolation, it suffices to prove (\ref{Estimate: First layer weighted Plancherel for limited alpha}) for $\alpha \in \mathbb{N}$. Therefore from Plancherel's theorem
\begin{align*}
   \int_G | |x|^{\alpha} \mathcal{K}_{F(\sqrt{\mathcal{L}})}(x,u) |^2 \, d(x,u) &= C \int_{\mathfrak{g}_{2,r}^{*}} \int_{\mathfrak{g}_1} |x|^{2 \alpha} \Big| \sum_{\mathbf{k} \in \mathbb{N}^{\Lambda}} F(\sqrt{\eta_{\mathbf{k}}^{\lambda}}) \prod_{n=1}^{\Lambda} \varphi_{k_n}^{(b_n^{\lambda}, r_n)}(R_{\lambda}^{-1} P_n^{\lambda}x) \Big|^2 dx \, d\lambda .
\end{align*}
Note that $P_n^{\lambda} R_{\lambda} = R_{\lambda} P_n$ for all $\lambda \in \mathfrak{g}_{2,r}^{*}$, where $P_n$ is the projection from $\mathbb{R}^{d_1}= \mathbb{R}^{2 r_1}\oplus \cdots \oplus \mathbb{R}^{2 r_{\Lambda}}$ onto $\mathbb{R}^n$ (see \cite[Proposition 3.1]{Niedorf_Metivier_Group_2025}). Making the change of variable $y= R_{\lambda}^{-1}x$ and then writing $|y|^{2\alpha} = (|P_1 y|^2 + \cdots + |P_{\Lambda} y|^2)^{\alpha}$, the right hand side of the above expression is dominated by
\begin{align}
\label{estimate: Reduced weighted Plancherel estimate I for limited alpha}
    C \int_{\mathfrak{g}_{2,r}^{*}} \int_{\mathbb{R}^{d_1}} \Big| \left(\prod_{n=1}^{\Lambda} |P_n y|^{m_n} \right) \sum_{\mathbf{k} \in \mathbb{N}^{\Lambda}} F(\sqrt{\eta_{\mathbf{k}}^{\lambda}}) \prod_{n=1}^{\Lambda} \varphi_{k_n}^{(b_n^{\lambda}, r_n)}(P_n y) \Big|^2 dy \, d\lambda ,
\end{align}
where $\sum_{i=1}^{\Lambda} m_i= \alpha$ with $m_i\in \mathbb{N}$ for $i=1, \ldots, \Lambda$.

From (6.3) of \cite[Theorem 6.1]{Niedorf_Metivier_Group_2025} we have the following sub-elliptic estimate: for $\beta \geq 0$, $m \in \mathbb{N} \setminus \{0\}$ and $\eta>0$
\begin{align*}
\| |\cdot|^{\beta} f\|_{L^2(\mathbb{R}^{2m})} &\lesssim \eta^{-\beta} \| (H_{\mathbb{R}^{2m}}^{(\eta)})^{\beta/2} f \|_{L^2(\mathbb{R}^{2m})}, \quad \text{where} \quad H_{\mathbb{R}^{2m}}^{(\eta)} = -\Delta_z + \tfrac{\eta^2}{4}|z|^2 \quad \text{on} \quad \mathbb{R}^{2m},
\end{align*}
and $\varphi_k^{(\mu,m)}$ are the eigen-functions of $H_{\mathbb{R}^{2m}}^{(\eta)}$ with eigenvalue $(2k+m)\mu$.

Using the above sub-elliptic estimate on every block of $\mathbb{R}^{d_1} = \mathbb{R}^{2 r_1} \oplus \cdots \oplus \mathbb{R}^{2 r_{\Lambda}}$, the term in (\ref{estimate: Reduced weighted Plancherel estimate I for limited alpha}) can be dominated by
    \begin{align}
    \label{estimate: Reduced weighted Plancherel estimate II for limited alpha}
        C \int_{\mathfrak{g}_{2,r}^{*}} \int_{\mathbb{R}^{d_1}} \Big| \sum_{\mathbf{k} \in \mathbb{N}^{\Lambda}} F(\sqrt{\eta_{\mathbf{k}}^{\lambda}}) \prod_{n=1}^{\Lambda} (b_n^{\lambda})^{-m_n} ((2k_n+r_n)b_n^{\lambda})^{m_n/2} \varphi_{k_n}^{(b_n^{\lambda}, r_n)}(P_n y) \Big|^2 dy \, d\lambda .
    \end{align}
    Now using orthogonality and $\|\varphi_{k_n}^{(b_n^{\lambda}, r_n)}\|_{L^2(\mathbb{R}^{2 r_n})}^2 \sim (b_n^{\lambda})^{r_n} (k_n+1)^{r_n-1}$ (see \cite[proof of Proposition 6.1]{Niedorf_Metivier_Group_2025}), (\ref{estimate: Reduced weighted Plancherel estimate II for limited alpha}) can be bounded by
    \begin{align}
    \label{estimate: Reduced weighted Plancherel estimate IV for limited alpha}
        C \sum_{\mathbf{k} \in \mathbb{N}^{\Lambda}} \int_{\mathfrak{g}_{2,r}^{*}} |F(\sqrt{\eta_{\mathbf{k}}^{\lambda}})|^2 \prod_{n=1}^{\Lambda} \frac{(2k_n+r_n)^{m_n}}{(b_n^{\lambda})^{m_n}} (b_n^{\lambda})^{r_n} (k_n+1)^{r_n-1} \, d\lambda .
    \end{align}
    Now changing $\lambda$ into polar coordinates, that is, $\lambda = \rho \omega$ with $\rho \in [0, \infty)$, $|\omega|=1$ and since $\lambda \mapsto b_n^{\lambda}$ are homogeneous of degree $1$, $\eta_{\mathbf{k}}^{\lambda} = \rho \eta_{\mathbf{k}}^{\omega}$ and again substituting $\rho = (\eta_{\mathbf{k}}^{\omega})^{-1} \eta$ in the inner integral, (\ref{estimate: Reduced weighted Plancherel estimate IV for limited alpha}) is estimated by constant times
    \begin{align}
    \label{Using polar cordinate and change of variable}
        & \int_{0}^{\infty} \int_{S^{d_2-1}} \sum_{\mathbf{k} \in \mathbb{N}^{\Lambda}} |F(\sqrt{\eta})|^2 \left[\prod_{n=1}^{\Lambda} \frac{(2k_n+r_n)^{m_n}}{(\eta_{\mathbf{k}}^{\omega})^{-m_n} \eta^{m_n} (b_n^{\omega})^{m_n}} \frac{\eta^{r_n}}{(\eta_{\mathbf{k}}^{\omega})^{r_n}} (b_n^{\omega})^{r_n} (k_n+1)^{r_n-1} \right] \frac{\eta^{d_2}}{(\eta_{\mathbf{k}}^{\omega})^{d_2}} \frac{d\eta}{\eta} d\sigma(\omega) \\
        &\nonumber = \int_{0}^{\infty} \int_{S^{d_2-1}} \sum_{\mathbf{k} \in \mathbb{N}^{\Lambda}} |F(\sqrt{\eta})|^2 \left[ \prod_{n=1}^{\Lambda} (2k_n+r_n)^{m_n+r_n-1} (\eta_{\mathbf{k}}^{\omega})^{m_n-r_n} (b_n^{\omega})^{r_n-m_n} \right] \frac{\eta^{Q/2-\alpha}}{(\eta_{\mathbf{k}}^{\omega})^{d_2}} \frac{d\eta}{\eta} \, d\sigma(\omega) .
    \end{align}
    Since $G$ is M\'etivier group, the endomorphism $J_{\lambda}$ is invertible for all $\lambda \in \mathfrak{g}_2^{*} \setminus \{0\}$. As $b_1^{\lambda}, \ldots, b_{\Lambda}^{\lambda}$ are the non-negative eigenvalues of $i J_{\lambda}$, this necessarily means that $b_n^{\lambda} \neq 0$ for all $\lambda \in \mathfrak{g}_2^{*} \setminus \{0\}$ (see \cite[Lemma 5]{Martini_Muller_New_Class_Two_Step_Stratified_Groups_2014}). Thus the image of $S^{d_2-1}$ under the continuous map $\lambda \mapsto \mathbf{b}^{\lambda}$ is compact subset of $(0, \infty)^{\Lambda}$. In particular 
    \begin{align*}
        b_n^{\lambda} \sim 1 \quad \text{for all}\quad \lambda \in S^{d_2-1}\ \text{and}\ n \in \{1, \ldots, {\Lambda}\}.
    \end{align*}
    Now using the fact $(|\mathbf{k}|+1) \sim \eta_{\mathbf{k}}^{\omega} \geq (2k_n+r_n) b_n^{\omega}$ and $b_n^{\omega} \sim 1$ for $2\alpha<d_2$, the last expression in \eqref{Using polar cordinate and change of variable} is further dominated by
    \begin{align*}
        & C \int_{0}^{\infty} \int_{S^{d_2-1}} \sum_{\mathbf{k} \in \mathbb{N}^{\Lambda}} |F(\sqrt{\eta})|^2 \eta^{Q/2-\alpha} \left[ \prod_{n=1}^{\Lambda} (\eta_{\mathbf{k}}^{\omega})^{m_n+r_n-1} (\eta_{\mathbf{k}}^{\omega})^{m_n-r_n} (b_n^{\omega})^{-2m_n+1} \right] (\eta_{\mathbf{k}}^{\omega})^{-d_2} \frac{d\eta}{\eta} \, d\sigma(\omega) \\
        &\lesssim \int_{0}^{\infty} \int_{S^{d_2-1}} \sum_{\mathbf{k} \in \mathbb{N}^{\Lambda}} |F(\eta)|^2 \eta^{Q-2\alpha} \left[ \prod_{n=1}^{\Lambda} (\eta_{\mathbf{k}}^{\omega})^{2m_n-1} \right] (\eta_{\mathbf{k}}^{\omega})^{-d_2} \frac{d\eta}{\eta} \, d\sigma(\omega) \\
        &\lesssim \int_{0}^{\infty} \int_{S^{d_2-1}} |F(\eta)|^2 \eta^{Q-2\alpha} \sum_{\mathbf{k} \in \mathbb{N}^{\Lambda}} (|\mathbf{k}|+1)^{2\alpha-N-d_2} \frac{d\eta}{\eta} \, d\sigma(\omega) \\
        &\lesssim \int_{0}^{\infty} |F(\eta)|^2 \eta^{Q-2\alpha} \frac{d\eta}{\eta} .
    \end{align*}
This completes the proof of the proposition.
\end{proof}

If, in addition, we take $F$ to be compactly supported in the above Proposition, then we get the following corollary. See also \cite[Theorem 3.2]{Hebisch_Zienkiewicz_Multiplier_Metivier_1995}, \cite[Proposition 3.2]{Bagchi_Molla_Singh_Bilinear_Metivier_JFA}.
\begin{corollary}
\label{Corollary: First layer weighted Planchrel dilated version}
    If $F: \mathbb{R} \to \mathbb{C}$ is a bounded Borel function supported in $[0, R]$, then for all $0\leq \alpha <d_2/2$ we have
    \begin{align*}
        \int_G | |x|^{\alpha} \mathcal{K}_{F(\sqrt{\mathcal{L}})}(x,u)|^2 \ d(x,u) &\leq C R^{Q-2\alpha} \|F(R \cdot)\|_{L^2}^2 .
    \end{align*}
\end{corollary}
\begin{proof}
From Proposition \ref{Proposition: First layer weighted Plancherel}, using the fact that $F$ is compactly supported in $[0,R]$ and $2\alpha-1<Q$ we obtain
\begin{align*}
    \int_G | |x|^{\alpha} \mathcal{K}_{F(\sqrt{\mathcal{L}})}(x,u) |^2 \ d(x,u) \leq C \int_{0}^{R} |F(\eta)|^2 \eta^{Q-2\alpha} \frac{d\eta}{\eta} \leq C R^{Q-2\alpha} \|F(R \cdot)\|_{L^2}^2 ,
\end{align*}
which completes the proof of the corollary.
\end{proof}

The following proposition is known as weighted Plancherel estimates for the sub-Laplacian on M\'etivier groups with extra first-layer weights. This type of estimates are useful in proving sharp spectral multiplier estimates, see \cite{Muller_Stein_Spectral_Multiplier_Heisenberg_Related_groups_1994}, \cite{Martini_Muller_New_Class_Two_Step_Stratified_Groups_2014}, \cite{Hebisch_Zienkiewicz_Multiplier_Metivier_1995} and references therein.
\begin{proposition}
\label{Prop: Weighted Plancherel using weight and distance}
If $F: \mathbb{R} \to \mathbb{C}$ is a bounded Borel function supported in $[0, R]$, then for all $\beta \geq 0$, all $\epsilon, R>0$, all $0\leq \alpha <d_2/2$ we have
    \begin{align*}
        \left(\int_G | (1+R\|(x,u)\|)^{\beta} (1+R|x|)^{\alpha} \mathcal{K}_{F(\sqrt{\mathcal{L}})}(x,u)|^2 \, d(x,u) \right)^{1/2} &\leq C R^{Q/2} \|F(R \cdot)\|_{L^2_{\beta+\epsilon}} .
    \end{align*}
\end{proposition}

\begin{proof}
From \cite[Theorem 6.1]{Martini_Sharp_Multiplier_Kohn_Laplacian_2017} for all $\beta \geq 0$, all $\epsilon, R>0$, and all $F: \mathbb{R} \to \mathbb{C}$ with $\supp F \subseteq [-R, R]$ we have
\begin{align*}
    \left(\int_G | (1+R\|(x,u)\|)^{\beta} \mathcal{K}_{F(\sqrt{\mathcal{L}})}(x,u)|^2 \, d(x,u) \right)^{1/2} &\leq C R^{\frac{Q}{2}} \|F(R \cdot)\|_{L^{\infty}_{\beta+\epsilon}} .
\end{align*}
Using the fact $|x| \leq \| (x,u) \|$ and Sobolev embedding, from the above estimate for all $\alpha \geq 0$ we obtain
\begin{align}
\label{Kernel estimate with distance}
    \left(\int_G | (1+R\|(x,u)\|)^{\beta} (1+R|x|)^{\alpha} \mathcal{K}_{F(\sqrt{\mathcal{L}})}(x,u)|^2 \, d(x,u) \right)^{1/2} &\leq C R^{\frac{Q}{2}} \|F(R \cdot)\|_{L^{2}_{\beta+\alpha+\frac{1}{2}+\epsilon}} .
\end{align}
Also from Corollary \ref{Corollary: First layer weighted Planchrel dilated version} for all $\alpha \in [0,d_2/2)$, and all $F: \mathbb{R} \to \mathbb{C}$ with $\supp F \subseteq [0, R]$, 
\begin{align}
\label{kernel estimate with weight}
    \left(\int_G | (1+R|x|)^{\alpha} \mathcal{K}_{F(\sqrt{\mathcal{L}})}(x,u)|^2 \, d(x,u) \right)^{1/2} &\leq C R^{\frac{Q}{2}} \|F(R \cdot)\|_{L^{2}} .
\end{align}
Consequently, interpolating \eqref{Kernel estimate with distance} and \eqref{kernel estimate with weight} (see \cite[proof of Lemma 1.2]{Mauceri_Meda_Multipliers_Stratified_groups_1990}) we obtain the required result.    
\end{proof}

Finally, we prove some weighted $L^1$- estimates of the convolution kernel $\mathcal{K}_{F(\sqrt{\mathcal{L}})}$. See also \cite[Theorem 3.11]{Martini_Lie_groups_Polynomial_Growth_2012}

\begin{proposition} \label{prop:L^1_estimate_of_kernel with distance}
If $F: \mathbb{R} \to \mathbb{C}$ is a bounded Borel function supported in $[0, R]$, then for all $\beta > Q/2-\alpha$, all $\epsilon, R>0$, all $0\leq \alpha <d_2/2$ and $\varepsilon \geq 0$ we have
\begin{align*}
\int_G |(1+R\|(x,u)\|)^{\varepsilon} \mathcal{K}_{F(\sqrt{\mathcal{L}})}(x,u)| \, d(x,u) &\leq C \|F(R \cdot)\|_{L^2_{\beta+\varepsilon+\epsilon}} .
\end{align*}    
\end{proposition}

\begin{proof}
Using H\"older's inequality we get
\begin{align*}
    & \int_G |(1+R\|(x,u)\|)^{\varepsilon} \mathcal{K}_{F(\sqrt{\mathcal{L}})}(x,u)| \, d(x,u) \leq C \left( \int_G \frac{d(x,u)}{(1+R\|(x,u)\|)^{2\beta} (1+R|x|)^{2\alpha}} \right)^{1/2} \\
    &\hspace{2cm} \times \left(\int_G | (1+R\|(x,u)\|)^{\beta+\varepsilon} (1+R|x|)^{\alpha} \mathcal{K}_{F(\sqrt{\mathcal{L}})}(x,u)|^2 \, d(x,u) \right)^{1/2} .
\end{align*}
Since $G$ is M\'etivier groups we always have $d_1>d_2$ (see \cite[Section 2]{Molla_Singh_Commutator_Metivier_Arxiv}). Therefore using Proposition \ref{Prop: Weighted Plancherel using weight and distance} and Lemma \ref{Lemma: Integration of distance and weight} for all $2(\alpha+\beta)>Q$, all $\epsilon, R>0$, all $0\leq \alpha <d_2/2$ we have
\begin{align*}
    \int_G |(1+R\|(x,u)\|)^{\varepsilon} \mathcal{K}_{F(\sqrt{\mathcal{L}})}(x,u)| \, d(x,u) &\leq C \|F(R \cdot)\|_{L^2_{\beta+\varepsilon+\epsilon}} .
\end{align*}    
Therefore, the proof of the proposition is completed.
\end{proof}


\section{Proof of main Theorem}
\label{Section: Proof of main theorem}

In this section, we show that our main Theorem \ref{Theorem: Metivier drift main theorem} can be deduced from the following three claims, which we prove in the next section.

\begin{claim} \label{Assumption 1}
Let $s>d/2$. Then for all even functions $F:\mathbb{R} \to \C$ with $\supp{\widehat{F}} \subseteq [-2,2]$, we have
\begin{align*} 
    \sup_{(y,s) \in B(0,1)} \, \int_{\|(x,u)\| \geq 2 \|(y,s)\|} \, |\mathcal{K}_{F(\sqrt{\mathcal{L}})}((x,u)(y,s)) - \mathcal{K}_{F(\sqrt{\mathcal{L}})}(x,u)| \, d(x,u) \leq C \, \|F\|_{L^2_{s, \sloc}} .
\end{align*}  
\end{claim}

\begin{claim}\label{Assumption 2}
Let $s>d/2$. Then for all even functions $F:\mathbb{R} \to \C$ with $\supp{\widehat{F}} \subseteq [-2,2]$, we have
\begin{align*}
    \sup_{0<r\leq 1} \, r \int_{\|(x,u)\| \geq r} \, |\mathcal{K}_{F(\sqrt{\mathcal{L}})}(x,u)| \, d(x,u) \leq C \, \|F\|_{L^2_{s, \sloc}} .
\end{align*}
\end{claim}

\begin{claim}
\label{Assumption 3}
Let $0\leq \beta<d_2/2$. Then for all $r \in (0, \infty)$ and all even function $F : \mathbb{R} \to \mathbb{C}$ such that $\supp{\widehat{F}} \subseteq [-r,r]$, we have
\begin{align*}
    \|\chi^{1/2} \mathcal{K}_{F(\sqrt{\mathcal{L}})}\|_{L^1} &\leq C e^{b_X r} \, (1+r)^{Q/2-\beta-1/2} \|(1+|\cdot|)^{Q/2-\beta-1/2} F\|_{L^{2}(\mathbb{R})} .
\end{align*}
\end{claim}

\medskip Before proceeding further, we define some notations.
For all $b \in \mathbb{R}^{+}$, we define
\begin{align*}
    \mathcal{B}_b &:= \{B((x,u),r) : (x,u) \in G , \  0<r \leq b \} .
\end{align*}
We also define the \emph{local doubling property}: For every $b >0$, there exists $D_b \in \mathbb{R}^{+}$ such that
\begin{align}
\label{Local doubling property}
    \mu_X(2B) \leq D_b \, \mu_X(B) \quad \quad \text{for all} \quad B \in \mathcal{B}_b .
\end{align}
We call $D_b$ to be the local doubling constant. It is known that $(G, \varrho, \mu_X)$ satisfies \eqref{Local doubling property} (see \cite[Lemma 2.3]{Martini_Ottazzi_Vallarino_spectral_multipliers_drift_2019}).

\medskip 
We use the following proposition in order to prove our main Theorem \ref{Theorem: Metivier drift main theorem}.

\begin{proposition}\cite[Proposition 2.7]{Martini_Ottazzi_Vallarino_spectral_multipliers_drift_2019}
\label{prop:h^1_to_L^1_boundedness}
If $T$ is a bounded operator on $L^2(d\mu_{X})$ and it has an integral kernel $\mathcal{K}_{T}$, which is locally integrable off diagonal and satisfies 
\begin{align}
I_{1}(T) & := \sup_{B \in \mathcal{B}_1} \, \sup_{(y, s), (z, r) \in B} \int_{(2B)^{c}} |\mathcal{K}_{T}((x,u), (y, s)) - \mathcal{K}_{T}((x,u), (z, r))| \, d\mu_{X}(x,u) \, < \infty, \\
\text{and  } \qquad  I_{2}(T) &:= \sup_{(y, s) \in G} \int_{B((y, s), 2)^{c}} |\mathcal{K}_{T}((x,u), (y, s))| \, d\mu_{X}(x,u) \, < \infty.
\end{align}
Then, $T$ is bounded from $\mathfrak{h}^1(d\mu_X)$ to $L^1(d\mu_X)$, and
$$
\|T\|_{\mathfrak{h}^1(d\mu_X) \to L^1(d\mu_X)} \leq \max \{I_{1}(T), I_{2}(T)\} + D_{2}^{1/2} \, \|T\|_{L^2(d\mu_X) \to L^2(d\mu_X)},
$$
where $D_2$ is the local doubling constant as defined in \eqref{Local doubling property}.
\end{proposition}

The reduction of the proof of Theorem \ref{Theorem: Metivier drift main theorem}, to the above-mentioned three claims is essentially contained in \cite{Martini_Ottazzi_Vallarino_spectral_multipliers_drift_2019}. For the readers convenience we write the proof here also. Note that in order to prove Theorem \ref{Theorem: Metivier drift main theorem}, by duality enough to consider the case $1 \leq p<2$. We shall prove the theorem for $p=1$, then using interpolation (see \cite[Proposition 3.8]{Martini_Ottazzi_Vallarino_spectral_multipliers_drift_2019}) we get it for $p\in (1,2)$.

\medskip We now prove Theorem \ref{Theorem: Metivier drift main theorem} for $p=1$, letting $F$ be as in the theorem. Therefore, we have to show that for $s>(d+1)/2$,
\begin{align}
\label{Condition to prove for p=1 case}
\|F(\mathfrak{D}_X)\|_{\mathfrak{h}^1(d\mu_X) \to L^1(d\mu_X)} \lesssim \, \|F\|_{\mathscr{L}_{s,\infty}^{2, -1/2 , b_X}} .
\end{align}
We define $\phi: \R \to \R$ to be an even smooth function supported in $[-1, 1]$ such that $\phi(u) = 1$ for all $u \in [-\tfrac{1}{4}, \tfrac{1}{4}]$ and 
$$
\sum_{k \in \mathbb{Z}} \phi(u-k) = 1, \qquad \text{for all} \quad u \in \R.
$$
Also, for $k \in \mathbb{N}$ and $k \geq 2$, define $\phi_{k}(u) = \phi(u-k+1) + \phi(u+ k-1)$. Then, support of $\phi_k$ is contained in $[k-2, k] \cup [-k, -k+2]$. 

\medskip Define a function $\eta: \R \to \R$ such that $\widehat{\eta} = \phi + \phi_2$ and decompose the function $F$ as $F = F_{l} + F_{g}$, where $F_{l} = \eta \ast F$. We now estimate the norms of $F_{l}(\mathfrak{D}_{X})$ and $F_{g}(\mathfrak{D}_{X})$ separately.

\subsection{Estimate of \texorpdfstring{$F_{l}$:}{}}
First note that from \cite[p. 1520]{Martini_Ottazzi_Vallarino_spectral_multipliers_drift_2019} we have
\begin{align}
\label{Non holomorphic to holomorphioc}
    \|F\|_{L_{s,\infty}^{2, -1/2}} \lesssim \|F\|_{\mathscr{L}_{s,\infty}^{2, -1/2 , b_X}} .
\end{align}
Hence, for $F_l$ enough to prove that, whenever $s>d/2$ we have
\begin{align}
\label{For Fl H1 L1 norm required to prove}
    \|F_{l}(\mathfrak{D}_X)\|_{\mathfrak{h}^1(d\mu_X) \to L^1(d\mu_X)} \lesssim \, \|F\|_{L^{2, -1/2}_{s, \infty}} .
\end{align}
In view of the Proposition \ref{prop:h^1_to_L^1_boundedness}, to prove \eqref{For Fl H1 L1 norm required to prove}, we have to show that 
\begin{align*}
    \max\{I_1(F_{l}(\mathfrak{D}_{X})), \, I_2(F_{l}(\mathfrak{D}_{X}))\} \lesssim \|F\|_{L^{2, -1/2}_{s, \infty}} \quad \text{and} \quad \|F_{l}(\mathfrak{D}_X)\|_{L^2(d\mu_X) \to L^2(d\mu_X)} \lesssim \|F\|_{L^{2, -1/2}_{s, \infty}} .
\end{align*}
Whenever $F \in L^{2, -1/2}_{s, \infty}$ and $s>1/2$, from \cite[eq. (3.3)]{Martini_Ottazzi_Vallarino_spectral_multipliers_drift_2019} yield
\begin{align*}
    \|F_{l}(\mathfrak{D}_X)\|_{L^2(d\mu_X) \to L^2(d\mu_X)} \lesssim \|F_l\|_{L^{\infty}} \lesssim \|F\|_{L^{2, -1/2}_{s, \infty}} .
\end{align*}
On the other hand, note that for $s>1/2$ from \cite[eq. (4.6) and (3.3)]{Martini_Ottazzi_Vallarino_spectral_multipliers_drift_2019}, we have
\begin{align*}
    \|F_l\|_{L^2_{s, \sloc}} &\lesssim \|F_l\|_{L^{2, -1/2}_{s, \infty}} \lesssim \|F\|_{L^{2, -1/2}_{s, \infty}} .
\end{align*}
Therefore in view of the above relation, it remains to show that
\begin{align}
\label{Required to prove in terms of I1 nad I2}
    \max\{I_1(F_{l}(\mathfrak{D}_{X})), \, I_2(F_{l}(\mathfrak{D}_{X}))\} \lesssim \, \|F_l\|_{L^2_{s, \sloc}} .
\end{align}

Recall that from \cite[p. 1507]{Martini_Ottazzi_Vallarino_spectral_multipliers_drift_2019} we have $\chi^{1/2} F_{l}(\mathfrak{D}_X)(\chi^{-1/2}f) = F_{l}(\sqrt{\mathcal{L}})f$. 
Using this, if we let $K_{F_{l}(\mathfrak{D}_X)}$ be the integral kernel of the operator $F_{l}(\mathfrak{D}_X)$ with respect to the measure $\mu_{X}$ and $\mathcal{K}_{F_{l}(\sqrt{\mathcal{L}})}$ be the convolution kernel of $F_{l}(\sqrt{\mathcal{L}})$ with respect to the Haar measure, then we have
\begin{align}
\label{Relation between two kernels in terms of chi}
    K_{F_{l}(\mathfrak{D}_X)}((x,u), (y,s)) &= \chi^{-1/2}((x,u) (y,s)) \, \mathcal{K}_{F_{l}(\sqrt{\mathcal{L}})}((y,s)^{-1}(x,u)).
\end{align}

As $F_{l} = \eta \ast F$, this implies $\widehat{F}_{l} = \widehat{\eta} \, \widehat{F}$ and hence $\supp{\widehat{F}_{l}} \subset [-2, 2]$. Since, the sub-Laplacian $\mathcal{L}$ posses finite speed propagation (see \cite[Lemma 4.4]{Niedorf_Metivier_Group_2025}), the kernel $\mathcal{K}_{F_{l}(\sqrt{\mathcal{L}})}$ is supported in $B(0,2)$ (see \cite[eq. (2.7)]{Martini_Ottazzi_Vallarino_spectral_multipliers_drift_2019}), which implies that $K_{F_{l}(\mathfrak{D}_X)}$ is supported in $\{((x,u), (y,s)): \, \varrho ((x,u), (y,s)) \leq 2 \}$. By using this, we obtain
\begin{align}
\label{Estimate of I2 part}
     I_{2}(F_{l}(\mathfrak{D}_X)) &= \sup_{(y, s) \in G} \int_{B((y, s),2)^{c}} |K_{F_{l}(\mathfrak{D}_X)}((x,u), (y,s))| \, d\mu_{X}(x,u) = 0 .
\end{align}

Now, we proceed to estimate $I_{1}(F_{l}(\mathfrak{D}_X))$. Let $B := B(c_{B}, r_{B})$ be such that $r_{B} \leq 1$. Using self-adjointness of $F_{l}(\mathfrak{D}_{X})$ along with  \eqref{Relation between two kernels in terms of chi} and the fact that $\chi$ is a character, we estimate
\begin{align}
\label{Estimate of I1 part first}
    & \sup_{(y, s), (z, r) \in B} \int_{(2B)^{c}} |K_{F_{l}(\mathfrak{D}_X)}((y, s), (x,u)) - K_{F_{l}(\mathfrak{D}_X)}((z, r), (x,u))| \, d\mu_{X}(x,u) \\
    &\nonumber\leq 2 \, \sup_{(y, s) \in B} \int_{(2B)^{c}} | \chi^{-1/2}((y,s)(x,u)) \, \mathcal{K}_{F_{l}(\sqrt{\mathcal{L}})}((x,u)^{-1}(y,s)) \\
    &\nonumber \hspace{50mm} - \chi^{-1/2}(c_{B}(x,u)) \, \mathcal{K}_{F_{l}(\sqrt{\mathcal{L}})}((x,u)^{-1} c_{B})| \, d\mu_{X}(x,u) \\
    &\nonumber = 2 \, \sup_{(y, s) \in B} \int_{(2B(0, r_{B}))^{c}} | \chi^{-1/2}((y,s) c_{B} (x,u)^{-1}) \, \mathcal{K}_{F_{l}(\sqrt{\mathcal{L}})}((x,u) c_{B}^{-1}(y,s)) \\
    &\nonumber \hspace{40mm} - \chi^{-1/2}(c_{B} c_{B} (x,u)^{-1}) \, \mathcal{K}_{F_{l}(\sqrt{\mathcal{L}})}(x,u)| \, \chi(c_{B} (x,u)^{-1}) \, d(x,u) \\
    &\nonumber = 2 \, \sup_{(y, s) \in B} \int_{(2B(0, r_{B}))^{c}} | \chi^{-1/2}( (x,u) c_{B}^{-1} (y,s)) \, \mathcal{K}_{F_{l}(\sqrt{\mathcal{L}})}( (x,u) c_{B}^{-1} (y,s)) \\
    &\nonumber \hspace{65mm} - \chi^{-1/2}(x,u) \, \mathcal{K}_{F_{l}(\sqrt{\mathcal{L}})}(x,u)| \, d(x,u) \\
    &\nonumber = 2 \, \sup_{(y, s) \in B(0, r_{B})} \int_{(2B(0, r_{B}))^{c}} | \chi^{-1/2}( (x,u) (y,s)) \, \mathcal{K}_{F_{l}(\sqrt{\mathcal{L}})}( (x,u) (y,s)) \\
    &\nonumber \hspace{65mm} - \chi^{-1/2}(x,u) \, \mathcal{K}_{F_{l}(\sqrt{\mathcal{L}})}(x,u)| \, d(x,u) .
\end{align}
Let $\widetilde{X}= (X_1, \ldots, X_{d_1})$. Then, for $(y, s) \in B(0, r_{B})$, by using triangle inequality, mean value theorem and the support of $\mathcal{K}_{F_{l}(\sqrt{\mathcal{L}})}$, we get 
\begin{align}
\label{Estimate of I1 part second}
    &\int_{(2B(0, r_{B}))^{c}} | \chi^{-1/2}( (x,u) (y,s)) \, \mathcal{K}_{F_{l}(\sqrt{\mathcal{L}})}( (x,u) (y,s)) - \chi^{-1/2}(x,u) \, \mathcal{K}_{F_{l}(\sqrt{\mathcal{L}})}(x,u)| \, d(x,u) \\
    &\nonumber \leq  \sup_{B(0, 3)} \chi^{-1/2}  \int_{(2B(0, r_{B}))^{c}} | \mathcal{K}_{F_{l}(\sqrt{\mathcal{L}})}((x,u) (y,s)) - \mathcal{K}_{F_{l}(\sqrt{\mathcal{L}})}(x,u)| \, d(x,u) \\
    &\nonumber \qquad \quad  + \sup_{B(0, 2)} \chi^{-1/2} \int_{(2B(0, r_{B}))^{c}} |\chi^{-1/2}(y, s) -1 | |\mathcal{K}_{F_{l}(\sqrt{\mathcal{L}})}(x,u)| \, d(x,u) \\
    &\nonumber \leq C  \sup_{B(0, 3)} \chi^{-1/2}  \int_{(2B(0, r_{B}))^{c}} | \mathcal{K}_{F_{l}(\sqrt{\mathcal{L}})}((x,u) (y,s)) - \mathcal{K}_{F_{l}(\sqrt{\mathcal{L}})}(x,u)| d(x,u) \\
    &\nonumber \qquad \quad + \sup_{B(0, 2)} \chi^{-1/2} \, \sup_{B(0,1)} |\widetilde{X} \chi^{-1/2}| \, r_B \int_{(2B(0, r_{B}))^{c}} |\mathcal{K}_{F_{l}(\sqrt{\mathcal{L}})}(x,u)| \, d(x,u) .
\end{align}
Therefore, combining \eqref{Estimate of I1 part first}, \eqref{Estimate of I1 part second}, using Claims \ref{Assumption 1} and \ref{Assumption 2} and definition of $I_{1}(F_{l}(\mathfrak{D}_{X}))$ we obtain
\begin{align*}
I_{1}(F_{l}(\mathfrak{D}_{X})) &\lesssim \|F_l\|_{L^2_{s, \sloc}} .
\end{align*}
The above estimate along with \eqref{Estimate of I2 part} completes the proof of \eqref{Required to prove in terms of I1 nad I2}. Consequently, we also get \eqref{For Fl H1 L1 norm required to prove} and then combining with the fact \eqref{Non holomorphic to holomorphioc} we obtain 
\begin{align}
\label{Final statement for the operator Fl}
    \|F_{l}(\mathfrak{D}_X)\|_{\mathfrak{h}^1(d\mu_X) \to L^1(d\mu_X)} &\lesssim \, \|F\|_{\mathscr{L}_{s,\infty}^{2, -1/2 , b_X}} .
\end{align}

\subsection{Estimate of \texorpdfstring{$F_{g}$:}{}}
Now, we shall work with $F_{g}$.
Decompose $F_{g} = \sum_{k \geq 3} T_{k}$, where $\widehat{T}_{k} = \phi_{k}  \widehat{F}$. Recall that, we have $s>(d+1)/2$. By using Claim \ref{Assumption 3} and \cite[Lemma 2.12]{Martini_Ottazzi_Vallarino_spectral_multipliers_drift_2019}, we get
\begin{align*}
    \|\chi^{1/2} \mathcal{K}_{T_{k}(\sqrt{\mathcal{L}})}\|_{L^{1}(d(x,u))} & \leq C\, e^{b_X k} \, k^{Q/2-\beta-1/2} \|(1+|\cdot|)^{Q/2-\beta-1/2}T_{k}\|_{L^{2}(\R)} \\
    & \leq C \, k^{Q/2-\beta-1/2-s} \|F\|_{\mathscr{L}_{s,2}^{2, Q/2-\beta-1/2-s , b_X}} .
\end{align*}
Thus using \eqref{Relation between two kernels in terms of chi}, above estimate and since $s > (d+1)/2$, we have 
\begin{align}
\label{L1 kernel estimate of KFg operator}
    \|K_{F_{g}(\mathfrak{D}_{X})}(\cdot, (y,s)) \|_{L^1(d\mu_{X}(x,u))} & = \|\chi^{1/2} \, \mathcal{K}_{F_{g}(\sqrt{\mathcal{L}})} \|_{L^1(d(x,u))} \leq \sum_{k \geq 3} \|\chi^{1/2} \, \mathcal{K}_{T_{k}(\sqrt{\mathcal{L}})} \|_{L^1(d(x,u))} \\
    &\nonumber \leq C \sum_{k \geq 3} k^{Q/2-\beta-1/2-s} \|F\|_{\mathscr{L}_{s,2}^{2, Q/2-\beta-1/2-s , b_X}} \\
    &\nonumber \leq C \|F\|_{\mathscr{L}_{s,2}^{2, Q/2-\beta-1/2-s , b_X}}.
\end{align}
On the other hand, since $0\leq \beta<d_2/2$ and $s>(d+1)/2$, we can choose $\beta$ such that $s>Q/2-\beta$, hence from \cite[p. 1520]{Martini_Ottazzi_Vallarino_spectral_multipliers_drift_2019} we obtain
\begin{align}
\label{Relation of complexified Norm}
    \|F\|_{\mathscr{L}_{s,2}^{2, Q/2-\beta-1/2-s , b_X}} \lesssim \|F\|_{\mathscr{L}_{s,\infty}^{2, -1/2 , b_X}} .
\end{align}
Therefore, combining \eqref{L1 kernel estimate of KFg operator} and \eqref{Relation of complexified Norm}, whenever $F \in \mathscr{L}_{s,\infty}^{2, -1/2 , b_X}$ we get $F_{g}(\mathfrak{D}_{X})$ is bounded on $L^1(d\mu_{X})$, which implies boundedness from $\mathfrak{h}^1(d\mu_X)$ to $L^1(d\mu_X)$, moreover
\begin{align}
\label{For Fg h1 L1 norm final estimate}
    \|F_{g}(\mathfrak{D}_X)\|_{\mathfrak{h}^1(d\mu_X) \to L^1(d\mu_X)} &\lesssim \|F\|_{\mathscr{L}_{s,\infty}^{2, -1/2 , b_X}} .
\end{align}

Since, $F = F_{l} + F_{g}$, using boundedness of $F_{l}(\mathfrak{D}_{X})$ and $F_{g}(\mathfrak{D}_{X})$, that is from \eqref{Final statement for the operator Fl} and \eqref{For Fg h1 L1 norm final estimate} we obtain the required conclusion \eqref{Condition to prove for p=1 case}.


\section{Proof of the Claims}
\label{Section: Proof of claims}
In this section, we prove the three claims: Claim \ref{Assumption 1}, Claim \ref{Assumption 2} and Claim \ref{Assumption 3} stated in Section \ref{Section: Proof of main theorem}. Let us start with proof of Claim \ref{Assumption 1}.

\subsection{Proof of Claim \ref{Assumption 1}}

Consider a function $\phi \in C_c^\infty(\tfrac{1}{2}, 2)$ such that $\sum_{j \in \Z} \phi(2^{-j}\lambda) = 1$ for all $\lambda > 0$. Set $F_j(\lambda) = F(\lambda) \phi(2^{-j} \lambda)$. Thus, we have $F(\lambda) = \sum_{j \in \Z} F_j(\lambda).$

Take any $(y,s) \in B(0,1)$ and choose $J$ to be the smallest integer such that $2^{J}\|(y,s)\| \geq 1$. Therefore we get
\begin{align*}
    & \int_{\|(x,u)\| \geq 2 \|(y,s)\|} \, |\mathcal{K}_{F(\sqrt{\mathcal{L}})}((x,u)(y,s)) - \mathcal{K}_{F(\sqrt{\mathcal{L}})}(x,u)| \, d(x,u) \\
    &\leq \sum_{j>J} \, \int_{\|(x,u)\|>2\|(y,s)\|} |\mathcal{K}_{F_j(\sqrt{\mathcal{L}})}((x,u)(y,s)) - \mathcal{K}_{F_j(\sqrt{\mathcal{L}})}(x,u)| \, d(x,u) \\
    & \hspace{2cm} + \sum_{j\leq J} \, \int_{\|(x,u)\|>2\|(y,s)\|} |\mathcal{K}_{F_j(\sqrt{\mathcal{L}})}((x,u)(y,s)) - \mathcal{K}_{F_j(\sqrt{\mathcal{L}})}(x,u)| \, d(x,u) \\
    &=: E_1 + E_2 .
\end{align*}
\noindent \textbf{Estimate of $E_1$:}
Note that if $\|(x,u)\|>2\|(y,s)\|$, then we have
\begin{align*}
    \|(x,u)(y,s)\| \geq \|(x,u)\|-\|(y,s)\| > \|(y,s)\| .
\end{align*}
Now since $\supp F_j \subseteq [2^{j-1}, 2^{j+1}]$, using Proposition \ref{prop:L^1_estimate_of_kernel with distance} for $s>d/2$ and the fact $2^{J}\|(y,s)\| \geq 1$, we obtain
\begin{align*}
   E_1 & \lesssim \sum_{j>J} \, \int_{\|(x,u)\|>\|(y,s)\|}  |\mathcal{K}_{F_j(\sqrt{\mathcal{L}})}(x,u)| \, d(x,u) \\
   & \lesssim \sum_{j>J} \, (1+2^{j} \|(y,s)\|)^{-\epsilon} \,\int_{\|(x,u)\|>\|(y,s)\|}  (1+ 2^{j}\|(x,u)\|)^{\epsilon} |\mathcal{K}_{F_j(\sqrt{\mathcal{L}})}(x,u)| \, d(x,u) \\
   & \lesssim \sum_{j>J} \, (1+2^{j} \|(y,s)\|)^{-\epsilon} \|F_j(2^j \cdot)\|_{L^2_{s}} \\ 
   & \lesssim (2^{J}\|(y,s)\|)^{-\epsilon} \, \|F\|_{L^2_{s,\sloc}} \lesssim \|F\|_{L^2_{s,\sloc}} .
\end{align*}

\noindent \textbf{Estimate of $E_2$:}
Let us set $G_j(\lambda) = F_j(\lambda) \, e^{2^{-2j} \lambda^2}$. Then we have
\begin{align*}
    \mathcal{K}_{F_j(\sqrt{\mathcal{L}})}(x,u) &= \mathcal{K}_{G_j(\sqrt{\mathcal{L}})} \ast \mathcal{K}_{e^{-2^{-2j}\mathcal{L}}}(x,u) .
\end{align*}
Further, we set $R_{(y,s)} f(x,u) := f((x,u)(y,s))$. Hence an application of Young's convolution inequality yields
\begin{align}
\label{Euivalently computiion kernel difference}
& \int_{\|(x,u)\|>2\|(y,s)\|} |\mathcal{K}_{F_j(\sqrt{\mathcal{L}})}((x,u)(y,s)) - \mathcal{K}_{F_j(\sqrt{\mathcal{L}})}(x,u)| \, d(x,u) \\
&\nonumber \leq \int_{G} \, |\mathcal{K}_{G_j(\sqrt{\mathcal{L}})} \ast \mathcal{K}_{e^{-2^{-2j}\mathcal{L}}}((x,u)(y,s)) - \mathcal{K}_{G_j(\sqrt{\mathcal{L}})} \ast \mathcal{K}_{e^{-2^{-2j}\mathcal{L}}}(x,u)| \, d(x,u) \\ 
&\nonumber  \leq \int_{G} \, \int_{G} |\mathcal{K}_{G_j(\sqrt{\mathcal{L}})}(z,s)|\,  | (R_{(y,s)} \mathcal{K}_{e^{-2^{-2j}\mathcal{L}}}-  \mathcal{K}_{e^{-2^{-2j}\mathcal{L}}})((z,s)^{-1} (x,u))| \, d(z,s) \, d(x,u) \\ 
&\nonumber  = \int_{G} \left(|\mathcal{K}_{G_j(\sqrt{\mathcal{L}})}| * | (R_{(y,s)} \mathcal{K}_{e^{-2^{-2j}\mathcal{L}}}-  \mathcal{K}_{e^{-2^{-2j}\mathcal{L}}})| \right)(x,u) \, d(x,u) \\
&\nonumber \leq \|\mathcal{K}_{G_j(\sqrt{\mathcal{L}})}\|_{L^1} \|R_{(y,s)} \mathcal{K}_{e^{-2^{-2j}\mathcal{L}}}-  \mathcal{K}_{e^{-2^{-2j}\mathcal{L}}}\|_{L^1} .
\end{align}
Now from \cite[Lemma 5.4, Proposition 4.2]{Martini_Ottazzi_Vallarino_Solvable_Extension_2018} we get
\begin{align}
\label{Mean value for difference of heat kernel}
\|R_{(y,s)} \mathcal{K}_{e^{-2^{-2j}\mathcal{L}}}-  \mathcal{K}_{e^{-2^{-2j}\mathcal{L}}}\|_{L^1} \leq \|(y,s)\| \|\widetilde{X} \mathcal{K}_{e^{-2^{-2j}\mathcal{L}}}\|_{L^1} \leq C \|(y,s)\| \, 2^{j} ,
\end{align}
where $\widetilde{X}= (X_1, \ldots, X_{d_1})$.

Since $\supp G_j \subseteq [2^{j-1}, 2^{j+1}]$, applying Proposition \ref{prop:L^1_estimate_of_kernel with distance} with $\varepsilon=0$ for $s>d/2$, we have
\begin{align}
\label{L1 kernel estimate of Gj}
\|\mathcal{K}_{G_j(\sqrt{\mathcal{L}})}\|_{L^1} &\leq C \|G_j(2^j \cdot) \|_{L^2_{s}} \leq C \|F_j(2^j \cdot) \|_{L^2_{s}} .
\end{align}
Therefore plugging the above two estimates \eqref{Mean value for difference of heat kernel}, \eqref{L1 kernel estimate of Gj} into \eqref{Euivalently computiion kernel difference} we obtain
\begin{align*}
E_2 &= \sum_{j\leq J} \, \int_{\|(x,u)\|>2\|(y,s)\|} |\mathcal{K}_{F_j(\sqrt{\mathcal{L}})}((x,u)(y,s)) - \mathcal{K}_{F_j(\sqrt{\mathcal{L}})}(x,u)| \, d(x,u) \\
&\lesssim \sum_{j\leq J} 2^{j} \, \|(y,s)\| \, \|F_j(2^j \cdot) \|_{L^2_{s}} \\
& \lesssim 2^{J} \|(y,s)\| \, \|F\|_{L^2_{s,\sloc}}  \lesssim \|F \|_{L^2_{s,\sloc}} .
\end{align*}  
Now combining the estimates of $E_1$ and $E_2$, we got the required Claim \ref{Assumption 1}.


\subsection{Proof of Claim \ref{Assumption 2}} 

Note that since $\supp \widehat{F} \subseteq [-2,2]$, due to the finite speed propagation we have $\supp \mathcal{K}_{F(\sqrt{\mathcal{L}})} \subseteq B(0, 2)$. Then decomposing the set $\{(x,u): 2 \geq \|(x,u)\|>r\}$ into the dyadic annuli, in order to prove Claim \ref{Assumption 2}, it is enough to show, for $s>d/2$ and all even functions $F:\mathbb{R} \to \C$ with $\supp{\widehat{F}} \subseteq [-2,2]$, we have
\begin{align}
\label{Reduced claim 2}
    \sup_{0<r \leq 1} \int_{r \leq \|(x,u)\| < 2r} |\mathcal{K}_{F(\sqrt{\mathcal{L}})}(x,u)| \, d(x, u) &\leq C \|F \|_{L^2_{s,\sloc}} .
\end{align}

Let $\phi$ be as in Claim \ref{Assumption 1}. For $0<r \leq 1$, suppose $J \in \mathbb{Z}$ be the smallest integer such that $2^{J} r >1$. Set $F_J(\lambda) = \sum_{j \leq J} F_j(\lambda)$. Then we write
\begin{align*}
    & \int_{r \leq \|(x,u)\| < 2r} |\mathcal{K}_{F(\sqrt{\mathcal{L}})}(x,u)| \, d(x, u) \\
    &\leq \int_{G} |\mathcal{K}_{F_J(\sqrt{\mathcal{L}})}(x,u)| \, d(x,u) + \sum_{j > J} \int_{r \leq \|(x,u)\| < 2r} |\mathcal{K}_{F_j(\sqrt{\mathcal{L}})}(x,u)| \, d(x,u) =: I_1 + I_2 .
\end{align*}

\noindent \textbf{Estimate of $I_1$:}
Note that $\supp F_J \subset [0, 2^{J+2}]$. Using proposition \ref{prop:L^1_estimate_of_kernel with distance} with $\varepsilon=0$ and $R = 2^{J+2}$, we get that for $s > d/2$,
$$
I_1 = \int_{G} |\mathcal{K}_{F_J(\sqrt{\mathcal{L}})}(x, u)| \, d(x,u) \leq C \|F_J(2^{J+2}\cdot)\|_{L^2_{s}} \leq C \|F\|_{L^2_{s, \sloc}} .
$$

\noindent \textbf{Estimate of $I_2$:}
In this case note $\supp F_j \subset [2^{j-1}, 2^{j+1}]$. Then from Proposition \ref{prop:L^1_estimate_of_kernel with distance} for $s>d/2$ we have
\begin{align*}
I_2  &\leq \sum_{j > J} (1+2^{j+1}r)^{-\epsilon} \int_{r \leq \|(x,u)\| < 2r} (1+2^{j+1}\|(x,u)\|)^{\epsilon} |\mathcal{K}_{F_j(\sqrt{\mathcal{L}})}(x,u)| \, d(x,u)  \\
& \lesssim \sum_{j > J} (1+2^{j+1}r)^{-\epsilon} \, \|F_j(2^{j+1}\cdot)\|_{L^2_{s}} \\
&\lesssim (2^{J}r)^{-\epsilon} \|F\|_{L^2_{s, \sloc}} \lesssim \|F\|_{L^2_{s, \sloc}} .
\end{align*}
Hence combining the estimates of $I_1$ and $I_2$, we get the required estimate \eqref{Reduced claim 2}.

\subsection{Proof of Claim \ref{Assumption 3}}

First note that, since $\supp{\widehat{F}} \subseteq [-r,r]$ due to the finite speed propagation of $\mathcal{L}$ and H\"older's inequality, we get
\begin{align}
\label{Use of Holder in claim 3}
    \|\chi^{1/2} \mathcal{K}_{F(\sqrt{\mathcal{L}})}\|_{L^1} &\leq \Big(\int_{B(0,r)} \frac{\chi(x,u)}{|x|^{2\beta}}\, dx du \Big)^{1/2} \| |x|^{\beta} \mathcal{K}_{F(\sqrt{\mathcal{L}})}\|_{L^2} .
\end{align}
Applying Proposition \ref{Proposition: First layer weighted Plancherel} with $0\leq \beta<d_2/2$, we have
\begin{align}
\label{Estimate of weighted Plancherel in claim 3}
    \| |x|^{\beta} \mathcal{K}_{F(\sqrt{\mathcal{L}})}\|_{L^2} &\leq C \left(\int_{0}^{\infty} |F(\lambda)|^2 \lambda^{Q-2\beta} \frac{d\lambda}{\lambda} \right)^{1/2} .
\end{align}
Note that $b_X>0$. From Lemma \ref{Lemma: Pointwise estimate of character} and for $r\geq 1$ we can see that
\begin{align}
\label{Estimate of integral with chi}
    \int_{B(0,r)} \frac{\chi(x,u)}{|x|^{2\beta}}\, dx du &= \int_{B(0,1)} \frac{\chi(x,u)}{|x|^{2\beta}}\, dx \, du + \int_{B(0,r)\setminus B(0,1)} \frac{\chi(x,u)}{|x|^{2\beta}}\, dx \, du \\
    &\nonumber \lesssim e^{2b_X} \int_{B(0,1)} \frac{dx\, du}{|x|^{2\beta}} + \sum_{j=1}^{\lfloor r \rfloor} e^{2b_X (j+1)} \int_{B(0,j+1)\setminus B(0,j)} \frac{dx\, du}{|x|^{2\beta}} .
\end{align}
Let us compute the integrals in the last inequality. For $j \geq 1$, set $A_j:= \{(x,u): j< |(x,u)| \leq j+1 \}$. Also let $A_{j}^{in}:= \{(x,u) \in A_j : |x|\leq j/2\}$ and $A_{j}^{out} := \{(x,u) \in A_j : |x|> j/2\}$. Then
\begin{align}
\label{Breaking the annulus into in and out}
    \int_{B(0,j+1)\setminus B(0,j)} \frac{dx\, du}{|x|^{2\beta}} &= \int_{A_{j}^{in}} \frac{dx\, du}{|x|^{2\beta}} + \int_{A_{j}^{out}} \frac{dx\, du}{|x|^{2\beta}} .
\end{align}
For the second integral in the right hand side of the above equality \eqref{Breaking the annulus into in and out}, note that
\begin{align}
\label{Integral of weight in in part}
    \int_{A_{j}^{out}} \frac{dx du}{|x|^{2\beta}} &\lesssim j^{-2\beta} |A_j| \lesssim j^{-2\beta} j^{Q-1} \lesssim j^{Q-2\beta-1} .
\end{align}
While for the first integral in the right hand side of \eqref{Breaking the annulus into in and out} we estimate as follows: if $|x| \leq j/2$ and $(x,u) \in A_j$ then $j^2/4 \leq |u| \leq (j+1)^2$. Consequently, for $0\leq \beta < d_1/2$,
\begin{align}
\label{Integral of weight in out part}
    \int_{A_{j}^{in}} \frac{dx\, du}{|x|^{2\beta}} &\lesssim \int_{|x| \leq j/2} \frac{dx}{|x|^{2\beta}} \int_{j^2/4 \leq |u| \leq (j+1)^2} du \lesssim j^{d_1-2\beta} j^{2d_2-1} \lesssim j^{Q-2\beta-1} .
\end{align}
Hence combining \eqref{Integral of weight in in part} and \eqref{Integral of weight in out part} for $0\leq \beta < d_1/2$ we have
\begin{align}
\label{Integral of weight inside annulus}
    \int_{B(0,j+1)\setminus B(0,j)} \frac{dx\, du}{|x|^{2\beta}} &\lesssim j^{Q-2\beta-1} .
\end{align}
On the other hand, for $0\leq \beta < d_1/2$,
\begin{align}
\label{Integral of weight inside ball of 1}
    \int_{B(0,1)} \frac{\chi(x,u)}{|x|^{2\beta}}\, dx \, du &\lesssim \int_{|x| \leq 1} \frac{dx}{|x|^{2 \beta}} \int_{|u|\leq 1} du \lesssim C .
\end{align}
Therefore plugging the above estimates \eqref{Integral of weight inside annulus} and \eqref{Integral of weight inside ball of 1} into \eqref{Estimate of integral with chi} for $0\leq \beta < d_1/2$ yields
\begin{align}
\label{Final integral of weight in claim 3}
    \int_{B(0,r)} \frac{\chi(x,u)}{|x|^{2\beta}}\, dx du  &\lesssim \sum_{j=0}^{\lfloor r \rfloor} e^{2b_X (j+1)} (j+1)^{Q-2\beta-1} \\
    &\nonumber \lesssim r^{Q-2\beta-1} \sum_{j=1}^{\lceil r \rceil} e^{2b_X j} \lesssim r^{Q-2\beta-1} e^{2b_X r} .
\end{align}
Recall that, on M\'etivier group we have $d_1>d_2$ (see \cite[Section 2]{Molla_Singh_Commutator_Metivier_Arxiv}). Finally, putting the estimates \eqref{Estimate of weighted Plancherel in claim 3} and \eqref{Final integral of weight in claim 3} into \eqref{Use of Holder in claim 3} for $0\leq \beta<d_2/2$, we obtain
\begin{align*}
    \|\chi^{1/2} \mathcal{K}_{F(\sqrt{\mathcal{L}})}\|_{L^1} &\leq C\, (1+r)^{Q/2-\beta-1/2} e^{b_X r} \left(\int_{0}^{\infty} |F(\lambda)|^2 \lambda^{Q-2\beta} \frac{d\lambda}{\lambda} \right)^{1/2} .
\end{align*}
Therefore, the proof of Claim \ref{Assumption 3} is completed.

\section*{Acknowledgments}
The first author is supported by the Senior Research Fellowship from the Indian Institute of Science Education and Research Bhopal. The second author was partially supported by the Prime Minister's Research Fellowship (PMRF) at the Indian Institute of Science Education and Research Kolkata as well as by an Institute Postdoctoral Fellowship at the Indian Institute of Science Education and Research Mohali. We would like to thank Rahul Garg for his careful reading of the manuscript and numerous useful suggestions.


\providecommand{\bysame}{\leavevmode\hbox to3em{\hrulefill}\thinspace}
\providecommand{\MR}{\relax\ifhmode\unskip\space\fi MR }
\providecommand{\MRhref}[2]{%
  \href{http://www.ams.org/mathscinet-getitem?mr=#1}{#2}
}
\providecommand{\href}[2]{#2}

\end{document}